\documentclass[11pt]{article}
\usepackage{marginnote}

\usepackage[left=1in,right=1in,top=0.98in,bottom=0.98in]{geometry}
     
\usepackage{hyperref}
\usepackage{amsmath, amsthm}
\usepackage{cleveref}
    \crefname{ex}{Example}{Examples}
    \crefname{lem}{Lemma}{Lemmas}
    \crefname{thm}{Theorem}{Theorems}
    \crefname{prop}{Proposition}{Propositions}
\usepackage[margin=1.5cm]{caption}

\let\hide\iffalse

\usepackage{tikz, graphicx}
\usepackage{subcaption}
    \newcommand{\hspfig}{\hspace{1cm}}
	\usetikzlibrary{positioning}
	\usetikzlibrary{arrows}
        \tikzset{%
        fwdrxn/.style={very thick, arrows={-Stealth[length=5pt,width=5pt]}},
        revrxn/.style={very thick, arrows={-Stealth[length=5pt,width=5pt,left]}}
        }
	
\usepackage{color, xcolor}
    
	\newcommand\blue[1]{{\textcolor{blue}{#1}}}
	\definecolor{orange}{RGB}{250, 140, 0}
		\newcommand\orange[1]{{\textcolor{orange}{#1}}}
	\definecolor{turq}{RGB}{0, 160, 160}
		
	\definecolor{violet}{RGB}{164, 98, 234}
		
	\definecolor{df}{RGB}{3, 133, 122}
	
\usepackage{amsthm}
	\newtheorem{thm}{Theorem}[section]
	\newtheorem{lem}[thm]{Lemma}
	\newtheorem{prop}[thm]{Proposition}
	\newtheorem{cor}[thm]{Corollary}
	\newtheorem*{thm*}{Theorem}
	\theoremstyle{definition}
		\newtheorem{defn}[thm]{Definition}
		\newtheorem{ex}[thm]{Example}

	\newtheoremstyle{TheoremNum}
        {\topsep}{\topsep}              
        {\itshape}                      
        {}                              
        {\bfseries}                     
        {.}                             
        { }                             
        {\thmname{#1}\thmnote{ \bfseries #3}}
    \theoremstyle{TheoremNum}

\newcommand{\df}[1]{{{\bf\emph{#1}}}}		

\usepackage[normalem]{ulem}	

\usepackage{parskip}
    \makeatletter
    \def\thm@space@setup{%
        \thm@preskip=\parskip \thm@postskip=0pt
    }\makeatother
\usepackage{amsfonts, amssymb, amsrefs, mathtools}
\usepackage[nocompress, noadjust]{cite} 
\newcommand{\eq}[1]{\begin{align*}#1\end{align*}}
	\newcommand{\eqn}[1]{\begin{align}#1\end{align}}  
	
\newcommand{\st}{\colon}                

\newcommand\mrm[1]{\mathrm{#1}}
 
\newcommand{\rr}{\ensuremath{\mathbb{R}}}

\renewcommand{\epsilon}{\varepsilon}	
\renewcommand{\phi}{\varphi}			

\DeclareMathOperator{\ran}{Im}	    	
\DeclareMathOperator{\Span}{span}		
\newcommand{\kk}{\kappa}

\newcommand{\vv}[1]{{\boldsymbol{#1}}}  
\newcommand{\mm}[1]{\mathbf{#1}}               
\newcommand{\rrp}{\rr_{\geq}}
\newcommand{\rrpp}{\rr_{>}}

\newcommand{\Gk}{\ensuremath{(G, \vv \kk)}}

\newcommand{\xx}{\vv x}
\newcommand{\yy}{\vv y}

\usepackage{enumitem}
\usepackage[version=4]{mhchem}
\usepackage{chemfig} \newcommand{\cf}[1]{\textsf{#1}}

\include{LineNoPackageOptions} 

\DeclareMathOperator{\codim}{codim}

\newcommand{\WRz}{$\mrm{WR}_0$}

\usepackage{algorithm,algorithmic}

\usepackage{authblk}
\title{
    Uniqueness of weakly reversible and \\deficiency zero realizations of dynamical systems
}
\author[1,2]{
         Gheorghe Craciun%
}
\author[1]{
        Jiaxin Jin%
}
\author[1]{
        Polly Y. Yu%
}
\affil[1]{\small Department of Mathematics, University of Wisconsin--Madison}
\affil[2]{\small Department of Biomolecular Chemistry, University of Wisconsin--Madison}
\date{} 

\begin{document}

\maketitle

\begin{abstract}
A reaction network together with a choice of rate constants uniquely gives rise to a system of differential equations, according to the law of mass-action kinetics. On the other hand, {\em different} networks can generate {\em the same} dynamical system under mass-action kinetics. Therefore, the problem of identifying ``the" underlying network of a dynamical system is {\em not well-posed, in general}. Here we show that the problem of identifying an underlying {\em weakly reversible deficiency zero network} is well-posed, in the sense that the solution is unique whenever it exists. This can be very useful in applications because from the perspective of both dynamics and network structure, a weakly reversibly deficiency zero ($\textit{WR}_\textit{0}$) realization is the simplest possible one. 
Moreover,  while mass-action systems can exhibit practically any dynamical behavior,  including multistability, oscillations, and chaos, \WRz\ systems are remarkably stable for {\em any} choice of rate constants: they have a unique positive steady state within each invariant polyhedron, and cannot give rise to oscillations or chaotic dynamics.  We also prove that both of our hypotheses (i.e., weak reversibility and deficiency zero) are necessary for uniqueness. 
\end{abstract}

\section{Introduction}
\label{sec:intro}

A common approach to mathematical modeling in biochemistry,  molecular biology, ecology, and (bio)chemical engineering is based on non-linear interactions between different species or entities, e.g.,  metabolic or signaling pathways in cells, predator-prey relations in population dynamics, and reactions inside a chemical reactor~\cite{Malthus1798, Verhulst1838, Feinberg1987, Sontag2001}. These interactions, represented by a directed graph and under specified kinetic rules, generate a system of differential equations that model the time-dependent abundance of the interacting species. The most common kind of such kinetic rules is \emph{mass-action kinetics}, where the rate at which interaction occurs is proportional to the abundance of interacting species, and the resulting differential equations have a polynomial right-hand side. Complex qualitative dynamics such as multistability, oscillations, and chaos are possible under mass-action kinetics~\cite{voit2015-150, yu2018mathematical}. In general, determining the qualitative dynamics of a given mass-action system can be a very difficult task; on the other hand, certain structures in an associated directed graph, or \emph{reaction network}, are linked to special dynamical properties.

One such network property is \emph{weak reversibility}, i.e., every reaction is part of an oriented cycle. A weakly reversible system has at least one positive steady state within each invariant polyhedron~\cite{Boros2019}. Even though some weakly reversible systems may have infinitely many steady states~\cite{boros2019weakly}, they are conjectured to be persistent and even permanent~\cite{CraciunNazarovPantea2013_GAC}. After some restrictions on rate constants, weakly reversible mass-action systems are \emph{complex-balanced}, which are known for admitting a  globally defined strict  Lyapunov function~\cite{HornJackson1972,feinberg1972complex, horn1972necessary}.

In particular, complex-balanced systems cannot give rise to multistability, oscillations, or chaotic dynamics, and are conjectured to have a globally stable positive steady state within each invariant polyhedron; this conjecture has been proved under some additional assumptions~\cite{CraciunNazarovPantea2013_GAC, Anderson2011_GAC, BorosHofbauer2019_GAC, gopalkrishnan2014geometric}. The number of constraints on the rate constants necessary for complex-balancing is measured by an integer called the \emph{deficiency} of the network~\cite{CraciunDickensteinShiuSturmfels2009} (see \Cref{def:deficiency}). In the case of weak reversibility and deficiency zero, which we call $\textbf{\emph{WR}}_\textbf{\emph{0}}$ for short, the system is {\em always} complex-balanced, regardless of the choice of rate constants; this is the statement of the Deficiency Zero Theorem~\cite{feinberg1972complex,horn1972necessary, feinberg2019foundations}.

An example of a \WRz\ biochemical system is a model for T-cell receptor signal transduction~\cite{Sontag2001, Mckeithan1995}, whereby a T-cell receptor ($\cf{T}$) forms an initial ligand-receptor complex ($\cf{C}_0$) that is converted (by phosphorylation) to an active form ($\cf{C}_N$) via a sequence of intermediates ($\cf{C}_i$); the active form $\cf{C}_N$ of the complex is responsible for generating a signal. Also, any of the ligand-receptor complex $\cf{C}_i$ may dissociate. The reaction network (with rate constants as labeled) proposed by McKeithan for this process is shown in \Cref{fig:T-cell}. This reaction network is weakly reversible and deficiency zero, and thus any positive steady state is asymptotically stable within its invariant polyhedron. Indeed, any positive steady state is \emph{globally stable} within its invariant polyhedron~\cite{Sontag2001}.

\begin{figure}[h!tbp]
	\centering
		\begin{tikzpicture}[scale=1]
		\node [opacity=0] at (0,0.5) {};
		\node [opacity=0] at (0,-0.5) {};
		    \node (i) at (-2,0) {$\cf{T}+\cf{M}$}; 
		    \node (0) at (0,0) {$\cf{C}_0$};
		    \node (1) at (1.5,0) {$\cf{C}_1$};
		    \node (2) at (3.25,0) {\phantom{\,\,}$\cdots$\phantom{\,\,}};
		    \node (3) at (5,0) {$\cf{C}_i$};
		    \node (4) at (6.75,0) {\phantom{\,\,}$\cdots$\phantom{\,\,}};
		    \node (5) at (8.5,0) {$\cf{C}_N$};
            \draw [-{stealth}, thick, transform canvas={yshift=0ex}] (i)--(0) node [midway, above] {\footnotesize $\kk_1$};
            \draw [-{stealth}, thick, transform canvas={yshift=0ex}] (0)--(1) node [midway, above] {\footnotesize $\kk_{p,0}$};
            \draw [-{stealth}, thick, transform canvas={yshift=0ex}] (1)--(2) node [midway, above] {\footnotesize $\kk_{p,1}$};
            \draw [-{stealth}, thick, transform canvas={yshift=0ex}] (2)--(3) node [midway, above] {\footnotesize $\kk_{p,i-1}$};
            \draw [-{stealth}, thick, transform canvas={yshift=0ex}] (3)--(4) node [midway, above] {\footnotesize $\kk_{p,i}$};
            \draw [-{stealth}, thick, transform canvas={yshift=0ex}] (4)--(5) node [midway, above] {\footnotesize $\kk_{p,N-1}$};
            \draw [-{stealth}, thick, transform canvas={yshift=0ex}] (0.south) to[out=270, in=330] node [midway, above] {\footnotesize $\kk_{-,0}$} (i.south east) ;
            \draw [-{stealth}, thick, transform canvas={yshift=0ex}] (1.south) to[out=270, in=330] node [near start, above left] {\footnotesize $\kk_{-,1}$\!\!\!} (-1.6,-0.3);
            
            \draw [-{stealth}, thick, transform canvas={yshift=0ex}] (3.south) .. controls (5,-1.2) and (4,-1.4) .. (3,-1.4)
            .. controls (2.5,-1.4) and (0,-1.5) .. (-1.9,-0.3); 
                \node at (4.3,-0.8) {\footnotesize $\kk_{-,i}$};
            
            \draw [-{stealth}, thick, transform canvas={yshift=0ex}] (5.south) .. controls (8.5,-1) and (8,-1.2) .. (7,-1.4)
            .. controls (6,-1.6) and (3,-1.7) .. (2,-1.6)
            .. controls (1.5,-1.55) and (0,-1.6) .. (-2.2,-0.3);
                \node at (7.5,-0.95) {\footnotesize $\kk_{-,N}$};
		\end{tikzpicture}
	\caption{A mass-action system modelling signal transduction of a T-cell receptor~\cite{Mckeithan1995}. A T-cell receptor $\cf{T}$ initially binds to a peptide-major histocompatibility complex $\cf{M}$, while the active form $\cf{C}_N$ of the ligand-receptor complex is responsible for signal generation. This \WRz\ system has a unique globally stable equilibrium up to conservation laws~\cite{Sontag2001}.}
	\label{fig:T-cell}
\end{figure}
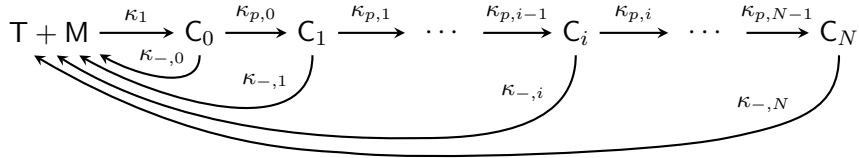

In short, under mass-action kinetics, a reaction network (dictating interactions of interest) and a choice of positive rate constants (proportionality constants for each interaction) uniquely determine the dynamics. Moreover, if the network satisfies certain conditions, then one can deduce the qualitative dynamics without solving the system of differential equations. 

On the other hand,  a given system of differential equations, even if known to have come from mass-action kinetics, is \emph{not} associated with a unique network structure~\cite{CraciunPantea2008}. Indeed, without additional requirements, there are infinitely many networks that can give rise to the same dynamical system under mass-action kinetics. Earlier studies took advantage of the possibility of finding a network with desirable properties, to conclude that a system of differential equations has certain dynamical properties~\cite{CraciunJinYu2019, CraciunJinYu_STN, JohnstonSiegelSzederkenyi2013, BrustengaCraciunSorea2020, CraciunSorea2020}. For example, the dynamical system
    \eq{ 
        \frac{dx}{dt} = 3 - 3x^3
    }
can be generated by any of the three networks in \Cref{fig:intro} using the rate constants labeled in the figure (and also can be generated by many other networks, for well-chosen rate constants). The network in \Cref{fig:intro-notWR} is neither weakly reversible nor has deficiency zero; the network in \Cref{fig:intro-notdef0} is weakly reversible but has a positive deficiency. The Deficiency Zero Theorem is therefore silent  until one recognizes that the same dynamics is also realized by the \WRz\ system in \Cref{fig:intro-def0}. While it is unnecessary to search for a \WRz\ realization for such simple differential equations, this process of finding \df{dynamically equivalent} (see \Cref{def:DE}) realizations applies to much more complicated and higher dimensional systems. Moreover, finding different realizations is essentially a linear feasibility problem, and algorithms exist for this very purpose~\cite{Szederkenyi2012_Toolbox, RudanSzederkenyiKatalinPeni2014}.

\renewcommand{\hspfig}{\hspace{1cm}}
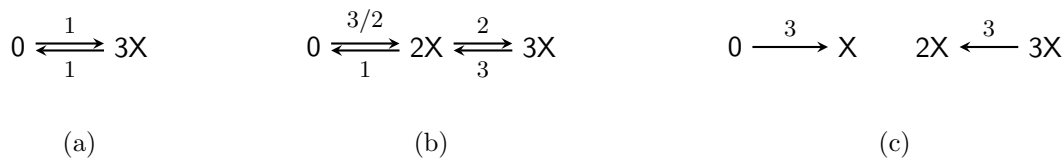
\begin{figure}[h!tbp]
	\centering
	\begin{subfigure}[t]{0.17\textwidth}
	\centering 
		\begin{tikzpicture}[scale=1.5]
		\node [opacity=0] at (0,0.5) {};
		\node [opacity=0] at (0,-0.5) {};
		    \node (0) at (0,0) {\cf{0}}; 
		    \node (3x) at (1,0) {\cf{3X}};
            \draw [-{stealth}, thick, transform canvas={yshift=0.28ex}] (0)--(3x) node [midway, above] {\footnotesize $1$};
            \draw [-{stealth}, thick, transform canvas={yshift=-0.28ex}] (3x)--(0) node [midway, below] {\footnotesize $1$};
		\end{tikzpicture}
		\caption{}
		\label{fig:intro-def0}
	\end{subfigure}
\hspfig
	\begin{subfigure}[t]{0.26\textwidth}
	\centering
		\begin{tikzpicture}[scale=1.5]
		\node [opacity=0] at (0,0.5) {};
		\node [opacity=0] at (0,-0.5) {};
		    \node (0) at (0,0) {\cf{0}}; 
		    \node (2x) at (1,0) {\cf{2X}};
		    \node (3x) at (2,0) {\cf{3X}};
            \draw [-{stealth}, thick, transform canvas={yshift=0.28ex}] (0)--(2x) node [midway, above] {\footnotesize $3/2$};
            \draw [-{stealth}, thick, transform canvas={yshift=0.28ex}] (2x)--(3x) node [midway, above] {\footnotesize $2$};
            \draw [-{stealth}, thick, transform canvas={yshift=-0.28ex}] (2x)--(0) node [midway, below] {\footnotesize $1$};
            \draw [-{stealth}, thick, transform canvas={yshift=-0.28ex}] (3x)--(2x) node [midway, below] {\footnotesize $3$};
		\end{tikzpicture}
		\caption{}
		\label{fig:intro-notdef0}
	\end{subfigure}
\hspfig
	\begin{subfigure}[t]{0.35\textwidth}
	\centering 
		\begin{tikzpicture}[scale=1.5]
		\node [opacity=0] at (0,0.5) {};
		\node [opacity=0] at (0,-0.5) {};
		    \node (0) at (0,0) {\cf{0}}; 
		    \node (x) at (1,0) {\cf{X}};
		    \node (2x) at (1.75,0) {\cf{2X}};
		    \node (3x) at (2.75,0) {\cf{3X}};
            \draw [-{stealth}, thick, transform canvas={yshift=0ex}] (0)--(x) node [midway, above] {\footnotesize $3$};
            \draw [-{stealth}, thick, transform canvas={yshift=0ex}] (3x)--(2x) node [midway, above] {\footnotesize $3$};
		\end{tikzpicture}
		\caption{}
		\label{fig:intro-notWR}
	\end{subfigure}
	\caption{Mass-action systems that give rise to the same system of differential equations. Labels on edges are rate constants for that reaction.}
	\label{fig:intro}
\end{figure}

The general non-uniqueness of networks that can generate a given dynamical system poses a challenge: what is a reaction mechanism that is consistent with kinetic data? The lack of network identifiability implies that even if one has experimental data with perfect accuracy and temporal resolution, it is impossible to determine the underlying reaction network, without imposing conditions on the network structure. Thus, it is reasonable to seek the \lq\lq simplest\rq\rq\ network that is consistent with the data. One may require minimal number of vertices in the network, or weak reversibility, or perhaps minimal deficiency; these are not unrelated. If there exists a weakly reversible realization, then there exists one that uses the minimal number of vertices, namely the distinct monomials appearing in the differential equations~\cite{CraciunJinYu2019}. Moreover, by minimizing the number of vertices in the network, one also minimizes deficiency (see \Cref{prop:V}).

Since \WRz\ realizations are desirable because of simplicity as well as stable dynamics (by the Deficiency Zero Theorem), the question we pose is: can there be {\em different} \WRz\ realizations for the same dynamical system? {\em The answer is no: $\textit{WR}_\textit{0}$ realizations are unique~(see \Cref{thm:WR}).} Moreover, we show that {\em both} weak reversibility and deficiency zero are necessary for uniqueness. 



This paper is organized as follows. \Cref{sec:mas} introduces mass-action systems and dynamical equivalence, as well as other notions and results of reaction network theory used throughout this work. Then we show the necessity of weak reversibility and deficiency zero in \Cref{sec:nec}, and describe restrictions on the network structure imposed by these two conditions in \Cref{sec:lem}. Finally, we prove the uniqueness of weakly reversible and deficiency zero realization in \Cref{sec:thm}.

\section{Background} 
\label{sec:mas}

In this section, we introduce mass-action systems and notions that are necessary for our exposition. For a short introduction to the mathematics of mass-action systems, see~\cite{yu2018mathematical}, and a detailed description of classical results,  see~\cite{feinberg2019foundations}. Throughout, let $\rrpp$ denote the set of positive real numbers, and $\rrpp^n$ the set of vectors with positive components, i.e., $\xx \in \rrpp^n$ if $x_i > 0$ for all $i = 1$, $2,\dots, n$. Analogously, let $\rrp$ and $\rrp^n$ denote the sets of non-negative numbers and vectors respectively. Summing over the empty set returns the zero vector, i.e., $\displaystyle \sum_{\yy \in \emptyset} \yy = \vv 0$. The disjoint union of sets is denoted $X \sqcup Y$. 

\bigskip

\begin{defn}
\label{def:crn}
	A \df{reaction network} is a directed graph $G = (V,E)$, where $V$ is a finite subset of $\rr^n$, and there are neither self-loops nor isolated vertices. 
\end{defn}

In the reaction network literature, a vertex is also called a \df{complex}. An edge $(\yy, \yy')$, also called a \df{reaction}, is denoted $\yy \to \yy'$. Vertices are points in $\rr^n$, so an edge $\yy \to \yy' \in E$ can be regarded as a bona fide vector between vertices. Each edge is associated to a \df{reaction vector} $\yy' - \yy \in \rr^n$.

The set of vertices of a reaction network can naturally be partitioned according to the connected components, also called \df{linkage classes} in the reaction network literature. Given a reaction network, we identify a linkage class by the subset of vertices in that connected component. A reaction network is said to be \df{weakly reversible} if every linkage class is strongly connected, i.e., every edge is part of an oriented cycle.

\begin{defn}
\label{def:affindep}
    A reaction network $(V,E)$ has \df{affinely independent linkage classes} if the vertices in each linkage class are affinely independent, i.e., if the vertices in $\{\yy_0, \yy_1,\ldots, \yy_m\} \subseteq V$ define a linkage class, then the set $\{ \yy_j - \yy_0 \st j=1,2,\ldots, m \}$ consists of linearly independent vectors.
\end{defn}

\begin{figure}[h!tpb]
\renewcommand{\hspfig}{\hspace{0cm}}
	\centering
	\begin{subfigure}[b]{0.3\textwidth}
	\centering 
		\begin{tikzpicture}
        \draw [step=1, gray!50!white, thin] (0,-0.3) grid (3.5,0.3);
		\node at (0,0.55) {};
		\node at (0,-0.5) {};
            \draw [->, gray] (0,0)--(3.75,0);
            \draw [gray] (0,-0.33)--(0,0.33);
            \node [inner sep=0pt] (0) at (0,0) {\blue{$\bullet$}};
            \node [inner sep=0pt] (3x) at (3,0) {\blue{$\bullet$}};
            \draw [-{stealth}, thick, blue, transform canvas={yshift=0.31ex}] (0)--(3x) node [midway, above] {\footnotesize $1$};
            \draw [-{stealth}, thick, blue, transform canvas={yshift=-0.31ex}] (3x)--(0) node [midway, below] {\footnotesize $1$};
		\end{tikzpicture}
		\caption{}
		\label{fig:intro-EG-CB}
	\end{subfigure}
		\hspfig
	\begin{subfigure}[b]{0.3\textwidth}
	\centering 
		\begin{tikzpicture}
        \draw [step=1, gray!50!white, thin] (0,-0.3) grid (3.5,0.3);
		\node at (0,0.55) {};
		\node at (0,-0.5) {};
            \draw [->, gray] (0,0)--(3.75,0);
            \draw [gray] (0,-0.33)--(0,0.33);
            \node [inner sep=0pt] (0) at (0,0) {\blue{$\bullet$}};
            \node [inner sep=0pt] (2x) at (2,0) {\blue{$\bullet$}};
            \node [inner sep=0pt] (3x) at (3,0) {\blue{$\bullet$}};
            \draw [-{stealth}, thick, blue, transform canvas={yshift=0.31ex}] (0)--(2x) node [midway, above] {\footnotesize $3/2$};
            \draw [-{stealth}, thick, blue, transform canvas={yshift=-0.31ex}] (2x)--(0) node [midway, below] {\footnotesize $1$};
            \draw [-{stealth}, thick, blue, transform canvas={yshift=0.31ex}] (2x)--(3x) node [midway, above] {\footnotesize $2$};
            \draw [-{stealth}, thick, blue, transform canvas={yshift=-0.31ex}] (3x)--(2x) node [midway, below] {\footnotesize $3$};
		\end{tikzpicture}
		\caption{}
		\label{fig:intro-EG-def1}
	\end{subfigure}
		\hspfig
	\begin{subfigure}[b]{0.3\textwidth}
	\centering 
		\begin{tikzpicture}
        \draw [step=1, gray!50!white, thin] (0,-0.3) grid (3.5,0.3);
		\node at (0,0.55) {};
		\node at (0,-0.5) {};
            \draw [->, gray] (0,0)--(3.75,0);
            \draw [gray] (0,-0.33)--(0,0.33);
            \node [inner sep=0pt] (0) at (0,0) {\blue{$\bullet$}};
            \node [inner sep=0pt] (x) at (1,0) {\blue{$\bullet$}};
            \node [inner sep=0pt] (2x) at (2,0) {\blue{$\bullet$}};
            \node [inner sep=0pt] (3x) at (3,0) {\blue{$\bullet$}};
            \draw [-{stealth}, thick, blue, transform canvas={yshift=0ex}] (0)--(x) node [midway, above] {\footnotesize $3$};
            \draw [-{stealth}, thick, blue, transform canvas={yshift=-0ex}] (3x)--(2x) node [midway, above] {\footnotesize $3$};
		\end{tikzpicture}
		\caption{}
		\label{fig:intro-EG-notWR}
	\end{subfigure}	
	\caption{The mass-action systems of \Cref{fig:intro} as embedded into the Euclidean space $\rr$. Labels on edges are rate constants.}
	\label{fig:intro-EG}
\end{figure}
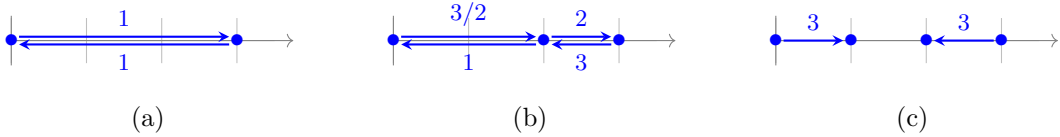

\begin{ex}
\label{ex:intro}
    The examples of \Cref{fig:intro} are represented as directed graphs embedded in $\rr$ in \Cref{fig:intro-EG}. The complexes $\cf{0}$ and $\cf{3X}$ correspond to the points $0$ and $3 \in \rr$ respectively. Similarly, the complexes $\cf{X}$ and $\cf{2X}$ can be understood as $1$ and $2 \in \rr$. The rate constants have been kept for clarity. An edge in a network is a vector that points from the source to the target. 
    
    The networks shown in \Cref{fig:intro-EG-CB,fig:intro-EG-def1} each has one linkage class, while that of \Cref{fig:intro-EG-notWR} has two linkage classes and is not weakly reversible. The network in \Cref{fig:intro-EG-def1} does \emph{not} have affinely independent linkage classes, since $0$, $2$, and $3$ are co-linear, thus not affinely independent, points of $\rr$. 
\end{ex}

\begin{defn}
\label{def:mas}
    Let $G = (V,E)$ be a reaction network in $\rr^n$. Let $\vv\kk \in \rrpp^{E}$ be a vector of \df{rate constants}. A \df{mass-action system} $\Gk$ is the weighted directed graph, whose \df{associated dynamical system} is the system of differential equations on $\rrpp^n$ 
    \eqn{\label{eq:mas}
        \frac{d\xx}{dt} = \sum_{\yy_i \to \yy_j \in E}\!\! \kk_{ij} \xx^{\yy_i}(\yy_j - \yy_i),
    }
    where $\xx^\yy = x_1^{y_1}x_2^{y_2}\cdots x_n^{y_n}$. 
\end{defn}

It is sometimes convenient to refer to $\kk_{ij}$ even though $\yy_i \to \yy_j$ may not be an edge in the network. In such cases, the convention is to take $\kk_{ij} = 0$. 

We can rearrange the sum on the right-hand side of \eqref{eq:mas} by grouping terms with the same monomial, each being multiplied by the weighted sum of reaction vectors originating from the corresponding source vertex, as in
    \eq{ 
        \frac{d\xx}{dt} = \sum_{\yy_i \in V} \,\xx^{\yy_i} \sum_{\yy_j \in V} \kk_{ij} (\yy_j - \yy_i).
    }
We give a name to the weighted sum associated with each monomial.

\begin{defn}
\label{def:directvect}
    Let $(G,\vv\kk)$ be a mass-action system, and let $\yy_i \in V$. The \df{net reaction vector from $\yy_i$} is 
    \eqn{ \label{eq:directvect}
        \sum_{\yy_j \in V} \kk_{ij} (\yy_j - \yy_i).
    }
\end{defn}
For convenience, we may refer to the net reaction vector even though $\yy_i \not\in V$; in this case, let the net reaction vector be the zero vector. 

The right-hand side of \eqref{eq:mas} clearly lies in the linear space of all net reaction vectors. It also lies in the (possibly larger)  \df{stoichiometric subspace}
    \eq{
        S = \Span\{ \vv y_j - \vv y_i \st \yy_i \to \yy_j \in E\} .
    }
Hence, any solution to \eqref{eq:mas} is confined to translate of $S$. For any $\xx_0 \in \rrpp^n$, the \df{stoichiometric class of $\xx_0$} is the polyhedron $(\xx_0+S)\cap \rrpp^n$. In the current work, we extend the notion of the stoichiometric subspace to subsets of vertices, usually from the same linkage class.

\begin{defn}
\label{def:stoichsubsp}
    Let $(V, E)$ be a reaction network and $V_0 \subseteq V$. The \df{stoichiometric subspace defined by $V_0$} is the vector space 
        \eq{
            S(V_0) = \Span \{ \yy_j - \yy_i \st  \yy_i, \,\yy_j \in V_0\}.
        }
\end{defn}
If $V_0 = \{ \yy_0,\yy_1,\ldots, \yy_m\}$, then $S(V_0) = \Span \{\yy_j - \yy_0 \st j=1,2,\ldots, m \}$, so that if $V_0$ consists of affinely independent vertices, the set of vectors from a fixed vertex to all others forms a basis for $S(V_0)$. Clearly, if $V_0 \subseteq V_1$, then $S(V_0) \subseteq S(V_1)$. If $V_0$ defines a linkage class of $G$, then $S(V_0)$, the stoichiometric subspace \emph{of the linkage class}, captures the geometry of this connected component. Furthermore, the stoichiometric subspace \emph{of the network} is the vector sum of the stoichiometric subspaces of the linkage classes. In particular, if $G$ has a single linkage class, then $S(V) = S$. However, if $G$ has multiple linkage class, $S$ may be a proper subspace of $S(V)$.

While weak reversibility is a property of a reaction network, it has dynamical implications. For example, every weakly reversible mass-action system has at least one positive steady state within every stoichiometric class~\cite{Boros2019}; they are also conjectured to be persistent and permanent, with certain cases proven, e.g., when $n = 2$~\cite{CraciunNazarovPantea2013_GAC}, or when $\dim S = 2$ with all trajectories bounded~\cite{Pantea2012_GAC}, or when there is only one linkage class~\cite{BorosHofbauer2019_GAC, Anderson2011_GAC}. Weak reversibility is also necessary for complex-balancing, known for their asymptotic stability and conjectured to be globally stable~\cite{Horn1974_GAC, Craciun2019_GAC-inclusion}. By definition, a positive state $\xx$ is a \df{complex-balanced steady state} of a mass-action system $(G, \vv\kk)$ if at every $\yy_i \in V$, we have 
    \eq{ 
        \sum_{\yy_i \to \yy_j \in E}\!\! \kk_{ij} \xx^{\yy_i} 
        = \sum_{\yy_j \to \yy_i \in E}\!\! \kk_{ji} \xx^{\yy_j} . 
    }
The above equation can be interpreted as balancing the fluxes across the vertex $\yy_i$. Once a mass-action system admits a complex-balanced steady state, all of its positive steady states are complex-balanced~\cite{HornJackson1972}; therefore we call such a mass-action system a \df{complex-balanced system}. Because every stoichiometric class has exactly one complex-balanced steady state, the set of positive steady states has dimension equal to $\dim S^\perp$~\cite{HornJackson1972}.

Not every weakly reversible mass-action system is complex-balanced. In general, for a system to be complex-balanced, the rate constants should satisfy some algebraic constraints, the number of which is measured by a non-negative integer called the deficiency~\cite{CraciunDickensteinShiuSturmfels2009}.  

\begin{defn}
\label{def:deficiency}
    Let $(V,E)$ be a reaction network with $\ell$ linkage classes and stoichiometric subspace $S$. The \df{deficiency} of the network is the integer $\delta = |V| - \ell - \dim S$. 
\end{defn}

One can also consider the \df{deficiency of a linkage class} given by $V_i \subseteq V$, defined as $\delta_i = |V_i| - 1 - \dim S(V_i)$. It is easy to see that 
    \eq{ 
        \delta \geq \sum_{i=1}^\ell \delta_i,
    }
with equality if and only if the stoichiometric subspaces of the linkage classes $\{S(V_i)\}_{i=1}^\ell$ are linearly independent. If $\delta = 0$, then necessarily $\delta_i = 0$ for all $i=1$, $2,\ldots, \ell$.

When the deficiency of a weakly reversible reaction network is zero, then for any choice of positive rate constants, the mass-action system is complex-balanced~\cite{horn1972necessary, feinberg1972complex}. It then follows that within every stoichiometric class, the associated dynamical system \eqref{eq:mas} has a unique positive steady state, which is linearly stable~\cite{HornJackson1972, Johnston_note}. The deficiency is a property of the reaction network, not of the associated dynamical system, yet in the case of deficiency zero, it has strong implications on the dynamics under mass-action kinetics. In this work, we are interested in weakly reversible and deficiency zero reaction networks, which we refer to as a \df{$\textbf{WR}_\textbf{0}$ network}.  

The associated dynamical system \eqref{eq:mas} is uniquely defined by the network and its rate constants; however, different reaction networks can give rise to the same system of differential equations under mass-action kinetics~\cite{CraciunPantea2008}. Dynamical equivalence captures the notion when different mass-action systems (networks with their rate constants) have the same associated dynamical system, which occurs if and only if the net reaction vectors coincide.

\begin{defn}
\label{def:DE}
    Two mass-action systems $(G,\vv\kk)$ and $(G', \vv\kk')$ are said to be \df{dynamically equivalent} if for all $\yy_i \in V \cup V'$, we have 
    \eqn{ \label{eq:DE} 
        \sum_{\yy_i \to \yy_j \in E} \!\! \kk_{ij} (\yy_j - \yy_i) 
        = \sum_{\yy_i \to \yy_j \in E'} \!\! \kk'_{ij} (\yy_j - \yy_i).
    }
    We call such mass-action systems \df{realizations} of the associated dynamical system. 
\end{defn}
The sums above may be over an empty set, for example on the left-hand side of \eqref{eq:DE} if $\yy_i \not\in V$; in such cases, we take the sum to be the zero vector. This demonstrates that a realization may contain as many vertices as one would like, as long as it includes those corresponding to the monomials that appear in the dynamical system. 

Finally, consider when only a subset of vertices satisfies \eqref{eq:DE}.

\begin{defn}
\label{def:DEsubset}
    Let $(G,\vv\kk)$ and $(G', \vv\kk')$ be mass-action systems and $V_0 \subseteq V \cup V'$. The mass-action systems are said to be \df{dynamically equivalent on $V_0$} if for all $\yy_i \in V_0$, we have 
    \eq{ 
        \sum_{\yy_i \to \yy_j \in E} \!\! \kk_{ij} (\yy_j - \yy_i) 
        = \sum_{\yy_i \to \yy_j \in E'} \!\! \kk'_{ij} (\yy_j - \yy_i).
    }
\end{defn}

\section{On unique weakly reversible, deficiency zero realization}

The goal of this paper is to prove that for an associated dynamical system of a mass-action system, there is at most one weakly reversible, deficiency zero realization, which we call a {\WRz\ realization}. This problem has been solved when there is only one linkage class~\cite{Csercsik2011ParametricUO, CraciunJohnstonSzederkenyiTonelloTothYu2020}. Here, we consider the problem in full generality.  

In \Cref{sec:nec}, we demonstrate the necessity of weak reversibility and deficiency zero through examples. Then in \Cref{sec:lem}, we establish some restrictions on the network structure of a \WRz\ realization. Finally, we prove that weak reversibility and deficiency zero are sufficient for a unique realization in \Cref{sec:thm}.

\subsection{Necessary conditions for uniqueness of \WRz\ realization} 
\label{sec:nec}

To show that weak reversibility and deficiency zero are necessary for uniqueness, we consider three examples. The first, shown in \Cref{fig:NotWR}, involves two realizations with zero deficiency but are not weakly reversible. In \Cref{fig:NotD0-1,fig:NotD0-2} are weakly reversible and dynamically equivalent systems but (some) have positive deficiencies. \Cref{fig:NotD0-1} illustrates when deficiency is positive because of an affinely \emph{dependent} linkage class, while \Cref{fig:NotD0-2} illustrates when it is because of linear \emph{dependence} between stoichiometric subspaces of linkage classes. 


\begin{ex}
\label{ex:NotWR}
The dynamically equivalent systems with zero deficiency shown in \Cref{fig:NotWR} share the same associated dynamical system, namely, 
    \eq{ 
        \dot{x} &= 1 + 2y^2, \\
        \dot{y} &= 1 - 2y^2. 
    }
Neither of the mass-action systems is weakly reversible, demonstrating that weak reversibility is necessary for a unique \WRz\ realization. Note that the monomials appearing in the associated dynamical system must be source vertices in the realization.
\end{ex}

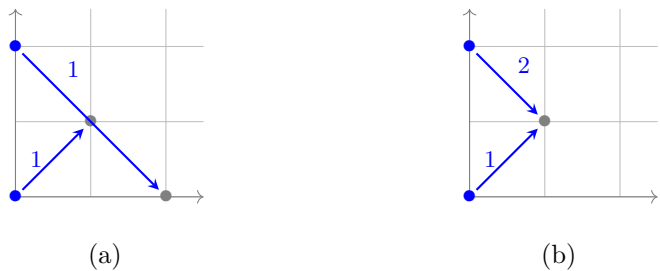
\begin{figure}[h!tbp]
	\centering
	\begin{subfigure}[b]{0.35\textwidth}
	\centering 
		\begin{tikzpicture}
        \draw [step=1, gray!50!white, thin] (0,0) grid (2.5,2.5);
        \node at (0,2.75) {};
		\node at (0,-0.25) {};
            \draw [->, gray] (0,0)--(2.5,0);
            \draw [->, gray] (0,0)--(0,2.5);
            \node [inner sep=0pt] (0) at (0,0) {\blue{$\bullet$}};
            \node [inner sep=0pt] (2y) at (0,2) {\blue{$\bullet$}};
            \node [inner sep=0pt, gray] (xy) at (1,1) {{$\bullet$}};
            \node [inner sep=0pt, gray] (2x) at (2,0) {{$\bullet$}};
            \draw [-{stealth}, thick, blue] (0)--(xy) node [midway, left] {\footnotesize $1$};
            \draw [-{stealth}, thick, blue] (2y)--(2x) node [near start, above right] {\footnotesize $1$};
		\end{tikzpicture}
		\caption{}
	\end{subfigure}
		\hspace{0cm}
	\begin{subfigure}[b]{0.35\textwidth}
	\centering
		\begin{tikzpicture}
        \draw [step=1, gray!50!white, thin] (0,0) grid (2.5,2.5);
        \node at (0,2.75) {};
		\node at (0,-0.25) {};
            \draw [->, gray] (0,0)--(2.5,0);
            \draw [->, gray] (0,0)--(0,2.5);
            \node [inner sep=0pt] (0) at (0,0) {\blue{$\bullet$}};
            \node [inner sep=0pt] (2y) at (0,2) {\blue{$\bullet$}};
            \node [inner sep=0pt, gray] (xy) at (1,1) {{$\bullet$}};
            \draw [-{stealth}, thick, blue] (0)--(xy) node [midway, left] {\footnotesize $1$};
            \draw [-{stealth}, thick, blue] (2y)--(xy) node [midway, above right] {\footnotesize $2$};
		\end{tikzpicture}
		\caption{}
	\end{subfigure}	
	\caption{Two dynamically equivalent mass-action systems with zero deficiency, but are not weakly reversible. Source vertices are shown in blue while targets are in gray.}
	\label{fig:NotWR}
\end{figure}


\begin{ex}
\label{ex:NotD0-1}
The dynamically equivalent, weakly reversible mass-action systems shown in \Cref{fig:NotD0-1} share the associated dynamical system 
    \eq{ 
        \dot{x} &= 2 - 2x^2y^2, 
        \\
        \dot{y} &= 2 - 2x^2y^2. 
    }
The network in \Cref{fig:NotD0-1b} has positive deficiency, demonstrating that $\delta = 0$ is necessary for a unique \WRz\ realization. The vertices in \Cref{fig:NotD0-1b} are \emph{not} affinely independent. As we shall see in \Cref{thm:def0}, this is one of two causes for a positive deficiency. 
\end{ex}

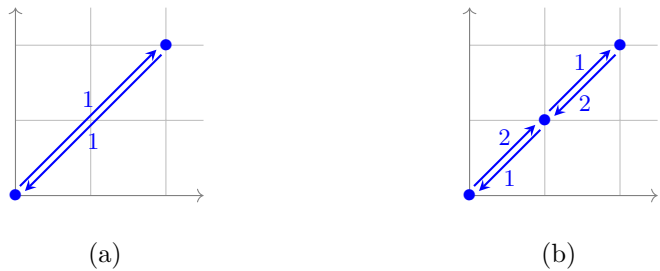
\begin{figure}[h!tbp]
	\centering
	\begin{subfigure}[b]{0.35\textwidth}
	\centering 
		\begin{tikzpicture}
        \draw [step=1, gray!50!white, thin] (0,0) grid (2.5,2.5);
        \node at (0,2.75) {};
		\node at (0,-0.25) {};
            \draw [->, gray] (0,0)--(2.5,0);
            \draw [->, gray] (0,0)--(0,2.5);
            \node [inner sep=0pt] (0) at (0,0) {\blue{$\bullet$}};
            \node [inner sep=0pt] (2x2y) at (2,2) {\blue{$\bullet$}};
            \draw [-{stealth}, thick, blue, transform canvas={xshift=-0.2ex, yshift=0.2ex}] (0)--(2x2y) node [midway, above] {\footnotesize $1$};
            \draw [-{stealth}, thick, blue, transform canvas={xshift=0.2ex, yshift=-0.2ex}] (2x2y)--(0) node [midway, below] {\footnotesize $1$};
		\end{tikzpicture}
		\caption{}
	\end{subfigure}
		\hspace{0cm}
	\begin{subfigure}[b]{0.35\textwidth}
	\centering
		\begin{tikzpicture}
        \draw [step=1, gray!50!white, thin] (0,0) grid (2.5,2.5);
        \node at (0,2.75) {};
		\node at (0,-0.25) {};
            \draw [->, gray] (0,0)--(2.5,0);
            \draw [->, gray] (0,0)--(0,2.5);
            \node [inner sep=0pt] (0) at (0,0) {\blue{$\bullet$}};
            \node [inner sep=0pt] (xy) at (1,1) {\blue{$\bullet$}};
            \node [inner sep=0pt] (2x2y) at (2,2) {\blue{$\bullet$}};
            \draw [-{stealth}, thick, blue, transform canvas={xshift=-0.2ex, yshift=0.2ex}] (0)--(xy) node [midway, above] {\footnotesize $2$};
            \draw [-{stealth}, thick, blue, transform canvas={xshift=0.2ex, yshift=-0.2ex}] (xy)--(0) node [midway, below] {\footnotesize $1$};
            \draw [-{stealth}, thick, blue, transform canvas={xshift=-0.2ex, yshift=0.2ex}] (xy)--(2x2y) node [midway, above] {\footnotesize $1$};
            \draw [-{stealth}, thick, blue, transform canvas={xshift=0.2ex, yshift=-0.2ex}] (2x2y)--(xy) node [midway, below] {\footnotesize $2$};
		\end{tikzpicture}
		\caption{}
		\label{fig:NotD0-1b}
	\end{subfigure}	
	\caption{Two dynamically equivalent mass-action systems that are weakly reversible, but the deficiency of the network in (b) is positive.}
	\label{fig:NotD0-1}
\end{figure}

\begin{ex}
Finally, the dynamically equivalent, weakly reversible mass-action systems shown in \Cref{fig:NotD0-2} share the associated dynamical system  
    \eq{ 
        \dot{x} &= 2 - x + x^2 - 2x^3. 
    }
However, both networks have a positive deficiency, demonstrating that $\delta = 0$ is necessary for a unique realization. In \Cref{fig:NotD0-2a}, the stoichiometric subspaces of the linkage classes are \emph{not} linearly independent. As we shall see in \Cref{thm:def0}, this is another cause for positive deficiency. 
\end{ex}

\begin{figure}[h!tbp]
	\centering
	\begin{subfigure}[b]{0.35\textwidth}
	\centering 
		\begin{tikzpicture}
		\draw [step=1, gray!50!white, thin] (0,-0.3) grid (3.5,0.3);
		\node at (0,0.55) {};
		\node at (0,-0.5) {};
            \draw [->, gray] (0,0)--(3.75,0);
            \draw [gray] (0,-0.33)--(0,0.33);
            \node [inner sep=0pt] (0) at (0,0) {\blue{$\bullet$}};
            \node [inner sep=0pt] (x) at (1,0) {\blue{$\bullet$}};
            \node [inner sep=0pt] (2x) at (2,0) {\blue{$\bullet$}};
            \node [inner sep=0pt] (3x) at (3,0) {\blue{$\bullet$}};
            \draw [-{stealth}, thick, blue, transform canvas={yshift=0.31ex}] (0)--(x) node [midway, above] {\footnotesize $2$};
            \draw [-{stealth}, thick, blue, transform canvas={yshift=-0.31ex}] (x)--(0) node [midway, below] {\footnotesize $1$};
            \draw [-{stealth}, thick, blue, transform canvas={yshift=-0.31ex}] (3x)--(2x) node [midway, below] {\footnotesize $2$};
            \draw [-{stealth}, thick, blue, transform canvas={yshift=0.31ex}] (2x)--(3x) node [midway, above] {\footnotesize $1$};
		\end{tikzpicture}
		\caption{}
		\label{fig:NotD0-2a}
	\end{subfigure}
		\hspace{0cm}
	\begin{subfigure}[b]{0.35\textwidth}
	\centering
		\begin{tikzpicture}
        \draw [step=1, gray!50!white, thin] (0,-0.3) grid (3.5,0.3);
		\node at (0,0.55) {};
		\node at (0,-0.5) {};
            \draw [->, gray] (0,0)--(3.75,0);
            \draw [gray] (0,-0.33)--(0,0.33);
            \node [inner sep=0pt] (0) at (0,0) {\blue{$\bullet$}};
            \node [inner sep=0pt] (x) at (1,0) {\blue{$\bullet$}};
            \node [inner sep=0pt] (2x) at (2,0) {\blue{$\bullet$}};
            \node [inner sep=0pt] (3x) at (3,0) {\blue{$\bullet$}};
            \draw [-{stealth}, thick, blue, transform canvas={yshift=0.7ex}] (0)--(2x) node [midway, above] {\footnotesize $1$};
            \draw [-{stealth}, thick, blue, transform canvas={yshift=-0.7ex}] (x)--(0) node [midway, below] {\footnotesize $1$};
            \draw [-{stealth}, thick, blue, transform canvas={yshift=0.7ex}] (2x)--(3x) node [midway, above] {\footnotesize $1$};
            \draw [-{stealth}, thick, blue, transform canvas={yshift=-0.7ex}] (3x)--(x) node [midway, below] {\footnotesize $1$};
		\end{tikzpicture}
		\caption{}
	\end{subfigure}	
	\caption{Two dynamically equivalent mass-action systems that are weakly reversible, but the deficiency of each network is positive.}
	\label{fig:NotD0-2}
\end{figure}
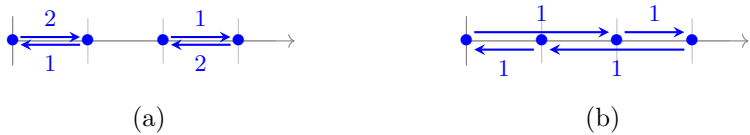

\subsection{Network structure of \WRz\ realizations} 
\label{sec:lem}

There are at least three restrictions on the network structure of a \WRz\ realization. First, deficiency zero can be characterized by affinely independent linkage classes and linearly independent stoichiometric subspaces of the linkage classes. Second, a \WRz\ realization uses the minimal number of vertices, precisely those corresponding to the monomials that appear explicitly in the dynamical system. Finally, if there are two \WRz\ realizations, they must have the same number of linkage classes. We address each of these assertions below.

If the deficiency $\delta$ of the network is zero, then the deficiency $\delta_i$ of each linkage class is also zero, and $\delta = \sum_i \delta_i$. The former implies that the stoichiometric subspace of the linkage class $S_i$ is of dimension $|V_i| - 1$, i.e., the vertices in $V_i$ are affinely independent. The latter implies that the stoichiometric subspaces of the linkage classes are linearly independent. Therefore, we have the following theorem.

\begin{thm}[\protect{\cite{feinberg2019foundations, CraciunJohnstonSzederkenyiTonelloTothYu2020}}]
\label{thm:def0}
    The deficiency of a reaction network is zero if and only if
    \begin{enumerate}[label={(\roman*)}]
    \item the network has affinely independent linkage classes, and
    \item the stoichiometric subspaces of the linkage classes are linearly independent.
    \end{enumerate}                                
\end{thm}

When it comes to weakly reversible realizations, any vertices for which the net reaction vector is zero can be removed while maintaining dynamical equivalence~\cite{CraciunJinYu2019}. Indeed, a weakly reversible realization exists if and only if one exists using only the vertices that appear in the monomials of the differential equations (after simplification). Moreover, for any additional vertex not coming from the monomials, the deficiency increases by one. 

\begin{prop}[\protect{\cite[Theorems 4.8 and 4.12]{CraciunJinYu2019}}]
\label{prop:V}
The vertices of any $\textit{WR}_\textit{0}$ realization of $\dot{\xx} = \vv f(\xx)$ are precisely the exponents in the monomials of $\vv f(\xx)$ after simplification. 
\end{prop}

Once the set of vertices is fixed, finding realizations (satisfying certain constraints like weak reversibility, minimal deficiency, or complex-balancing) is relatively simple, and algorithms based on optimization techniques exist. For example, see \cite{JohnstonSiegelSzederkenyi2013, RudanSzederkenyiKatalinPeni2014}, or \cite{Szederkenyi2012_Toolbox} for a MATLAB implementation.


Finally, we remark that the number of linkage classes in any \WRz\ realization is also fixed. Recall that deficiency of a network $G$ is $\delta_G = |V| - \ell_G - \dim S_G$, where $|V|$ is the number of vertices in $G$, $\ell_G$ is the number of linkage classes, and $S_G$ is the stoichiometric subspace. We already noted in \Cref{prop:V} that $|V|$ is the number of distinct monomials in the differential equations after simplification, hence constant between dynamically equivalent \WRz\ realizations. If $(G, \vv\kk)$ and $(G',\vv\kk')$ are two such realizations, they share the same associated dynamical system and thus the same set of positive steady states $Z$, whose codimension is $\dim S_G$~\cite{HornJackson1972}. Therefore, if $\delta_G = \delta_{G'} = 0$, then 
    \eq{ 
        \ell_G = |V| - \dim S_G = |V| - \codim Z = \ell_{G'}.
    }
In other words, the realizations have the same number of linkage classes. 

\begin{prop}\label{prop:ell}
In any $\textit{WR}_\textit{0}$ realization of $\dot{\xx} = \vv f(\xx)$, the number of linkage classes is given by $|V| - \codim Z$, where $|V|$ is the number of distinct monomials in $\vv f(\xx)$ after simplification, and $Z$ is the set of positive steady states. 
\end{prop}

\subsection{Proof of uniqueness of $\textbf{WR}_\textbf{0}$ realization}
\label{sec:thm}

The case for the uniqueness of \WRz\ realization with a single linkage class follows immediately from \Cref{thm:def0}. This case was first proved in \cite{Csercsik2011ParametricUO} using linear algebraic methods; a more geometric proof can be found in \cite{CraciunJohnstonSzederkenyiTonelloTothYu2020}.

\begin{cor} \label{cor:l=1} 
Any weakly reversible, deficiency zero realization with a single linkage class is unique.
\end{cor}

\begin{proof}
    Suppose that two \WRz\ realizations $(G, \vv\kk)$ and $(G', \vv\kk')$ are dynamically equivalent, i.e., for every $\yy_i \in V = V'$, we have
    \eq{ 
        \sum_{\yy_j \neq \yy_i} \kk_{ij} (\yy_j - \yy_i) 
        = \sum_{\yy_j \neq \yy_i} \kk'_{ij} (\yy_j - \yy_i) ,
    }
where $\kk_{ij}$, $\kk'_{ij} \geq 0$ are the appropriate rate constants from the realizations (which is set to zero if no edge is present). Because the vertices are affinely independent, the only solution to the linear equation
    \eq{ 
        \sum_{\yy_j \neq \yy_i} \left(\kk_{ij} - \kk'_{ij}\right)  (\yy_j - \yy_i) = \vv 0
    }
is $\kk_{ij} = \kk'_{ij}$ for all $i \neq j$. This is true for all edges with source vertex $\yy_i$. Hence, $(G,\vv\kk) = (G', \vv\kk')$.
\end{proof}

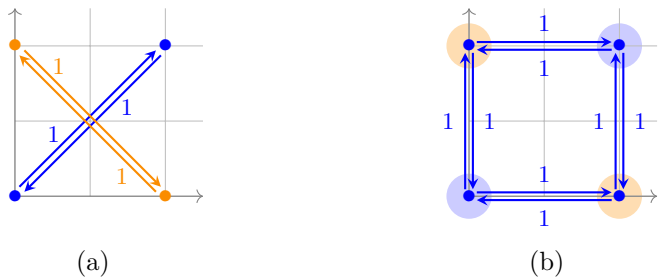
\begin{figure}[h!tbp]
	\centering
	\begin{subfigure}[b]{0.35\textwidth}
	\centering 
		\begin{tikzpicture}
		\draw [-{stealth}, thick, opacity=0, transform canvas={ yshift=-0.31ex}] (2x)--(0) node [midway, below] {\footnotesize $1$}; 
        \draw [-{stealth}, thick, opacity=0, transform canvas={ xshift=-0.31ex}] (0)--(2y) node [midway, left] {\footnotesize $1$};
        \draw [step=1, gray!50!white, thin] (0,0) grid (2.5,2.5);
        \node at (0,2.75) {};
		\node at (0,-0.25) {};
            \draw [->, gray] (0,0)--(2.5,0);
            \draw [->, gray] (0,0)--(0,2.5);
            \node [inner sep=0pt, blue] (0) at (0,0) {$\bullet$};
            \node [inner sep=0pt, blue] (2x2y) at (2,2) {$\bullet$};
            \draw [-{stealth}, thick, blue, transform canvas={xshift=-0.2ex, yshift=0.2ex}] (0)--(2x2y) node [near start, above] {\footnotesize $1$};
            \draw [-{stealth}, thick, blue, transform canvas={xshift=0.2ex, yshift=-0.2ex}] (2x2y)--(0) node [near start, below] {\footnotesize $1$};
            \node [inner sep=0pt, orange] (xy) at (2,0) {$\bullet$};
            \node [inner sep=0pt, orange] (2y) at (0,2) {$\bullet$};
            \draw [-{stealth}, thick, orange, transform canvas={xshift=0.2ex, yshift=0.2ex}] (2y)--(xy) node [near start, above] {\footnotesize $1$};
            \draw [-{stealth}, thick, orange, transform canvas={xshift=-0.2ex, yshift=-0.2ex}] (xy)--(2y) node [near start, below] {\footnotesize $1$};
		\end{tikzpicture}
		\caption{}
		\label{fig:contain-a}
	\end{subfigure}
		\hspace{0cm}
	\begin{subfigure}[b]{0.35\textwidth}
	\centering
		\begin{tikzpicture}
        \draw [fill = orange!30!white, draw opacity=0] (0,2) circle (0.3);
        \draw [fill = orange!30!white, draw opacity=0] (2,0) circle (0.3);
        \draw [fill = blue!20!white, draw opacity=0] (0,0) circle (0.3);
        \draw [fill = blue!20!white, draw opacity=0] (2,2) circle (0.3);
        \draw [step=1, gray!50!white, thin] (0,0) grid (2.5,2.5);
        \node at (0,2.75) {};
		\node at (0,-0.25) {};
            \draw [->, gray] (0,0)--(2.5,0);
            \draw [->, gray] (0,0)--(0,2.5);
            \node [inner sep=0pt, blue] (xy) at (2,0) {$\bullet$};
            \node [inner sep=0pt, blue] (2y) at (0,2) {$\bullet$};
            \node [inner sep=0pt, blue] (0) at (0,0) {$\bullet$};
            \node [inner sep=0pt, blue] (2x2y) at (2,2) {$\bullet$};
            \draw [-{stealth}, thick, blue, transform canvas={ yshift=0.31ex}] (0)--(2x) node [midway, above] {\footnotesize $1$};
            \draw [-{stealth}, thick, blue, transform canvas={ yshift=-0.31ex}] (2x)--(0) node [midway, below] {\footnotesize $1$}; 
            \draw [-{stealth}, thick, blue, transform canvas={ yshift=0.31ex}] (2y)--(2x2y) node [midway, above] {\footnotesize $1$};
            \draw [-{stealth}, thick, blue, transform canvas={ yshift=-0.31ex}] (2x2y)--(2y) node [midway, below] {\footnotesize $1$}; 
            \draw [-{stealth}, thick, blue, transform canvas={ xshift=-0.31ex}] (0)--(2y) node [midway, left] {\footnotesize $1$};
            \draw [-{stealth}, thick, blue, transform canvas={ xshift=0.31ex}] (2y)--(0) node [midway, right] {\footnotesize $1$}; 
            \draw [-{stealth}, thick, blue, transform canvas={ xshift=-0.31ex}] (2x)--(2x2y) node [midway, left] {\footnotesize $1$};
            \draw [-{stealth}, thick, blue, transform canvas={ xshift=0.31ex}] (2x2y)--(2x) node [midway, right] {\footnotesize $1$}; 
		\end{tikzpicture}
		\caption{}
		\label{fig:contain-b}
	\end{subfigure}	
	\caption{Two dynamically equivalent mass-action systems. Each linkage class of (a) is properly contained in a linkage class of (b). Vertices in (b) are highlighted according to the partitioning by the linkage classes of (a).}
	\label{fig:contain}
\end{figure}

In the case of multiple linkage classes, there could (in theory at least) be different ways to partition the vertex set by linkage classes while maintaining dynamical equivalence. For example, the two systems in \Cref{fig:contain} are dynamically equivalent and weakly reversible. While the linkage classes of (a) are 
    \eq{ 
        V_1 \sqcup V_2 = \{ \cf{0}, \, \cf{2X} + \cf{2Y} \} \sqcup \{ \cf{2X}, \, \cf{2Y} \}, 
    }
the network (b) has only one linkage class. Note that $V_1$ is properly contained in the linkage class of (b), which has affinely \emph{dependent} vertices and deficiency one. 


\begin{figure}[h!tbp]
	\centering
	\begin{subfigure}[b]{0.35\textwidth}
	\centering 
		\begin{tikzpicture}
        \draw [step=1, gray!50!white, thin] (0,-0.3) grid (3.5,0.3);
		\node at (0,0.55) {};
		\node at (0,-0.5) {};
            \draw [->, gray] (0,0)--(3.75,0);
            \draw [gray] (0,-0.33)--(0,0.33);
            \node [inner sep=0pt] (0) at (0,0) {\orange{$\bullet$}};
            \node [inner sep=0pt] (x) at (1,0) {\blue{$\bullet$}};
            \node [inner sep=0pt] (2x) at (2,0) {\blue{$\bullet$}};
            \node [inner sep=0pt] (3x) at (3,0) {\orange{$\bullet$}};
            \draw [-{stealth}, thick, orange, transform canvas={yshift=0.7ex}] (0)--(3x) node [near start, above] {\footnotesize $2$\hphantom{\qquad}};
            \draw [-{stealth}, thick, orange, transform canvas={yshift=-0.5ex}] (3x)--(0) node [near start, below] {\footnotesize \hphantom{\qquad}$2$};
            \draw [-{stealth}, thick, blue, transform canvas={yshift=1.2ex}] (x)--(2x) node [midway, above] {\footnotesize $2$};
            \draw [-{stealth}, thick, blue, transform canvas={yshift=-1.2ex}] (2x)--(x) node [midway, below] {\footnotesize $2$};
		\end{tikzpicture}
		\caption{}
		\label{fig:exchange-a}
	\end{subfigure}
		\hspace{0cm}
	\begin{subfigure}[b]{0.35\textwidth}
	\centering 
		\begin{tikzpicture}
		\draw [fill = orange!30!white, draw opacity=0] (0,0) circle (0.3);
		\draw [fill = orange!30!white, draw opacity=0] (3,0) circle (0.3);
        \draw [fill = blue!20!white, draw opacity=0] (1,0) circle (0.3);
        \draw [fill = blue!20!white, draw opacity=0] (2,0) circle (0.3);
        \draw [step=1, gray!50!white, thin] (0,-0.3) grid (3.5,0.3);
		\node at (0,0.55) {};
		\node at (0,-0.5) {};
            \draw [->, gray] (0,0)--(3.75,0);
            \draw [gray] (0,-0.33)--(0,0.33);
            \node [inner sep=0pt] (0) at (0,0) {\orange{$\bullet$}};
            \node [inner sep=0pt] (x) at (1,0) {\blue{$\bullet$}};
            \node [inner sep=0pt] (2x) at (2,0) {\orange{$\bullet$}};
            \node [inner sep=0pt] (3x) at (3,0) {\blue{$\bullet$}};
            \draw [-{stealth}, thick, orange, transform canvas={yshift=0.7ex}] (0)--(2x) node [near start, above] {\footnotesize $3$};
            \draw [-{stealth}, thick, orange, transform canvas={yshift=-0.5ex}] (2x)--(0) node [near end, below] {\footnotesize $1$};
            \draw [-{stealth}, thick, blue, transform canvas={yshift=1.2ex}] (x)--(3x) node [midway, above] {\footnotesize $1$};
            \draw [-{stealth}, thick, blue, transform canvas={yshift=-1.2ex}] (3x)--(x) node [midway, below] {\footnotesize $3$};
		\end{tikzpicture}
		\caption{}
		\label{fig:exchange-b}
	\end{subfigure}
	\caption{Two dynamically equivalent mass-action systems. No linkage class of~(a) is properly contained in any linkage class of~(b), and vice versa. Vertices in~(b) are highlighted according to the partitioning by the linkage classes of~(a).}
	\label{fig:exchange}
\end{figure}
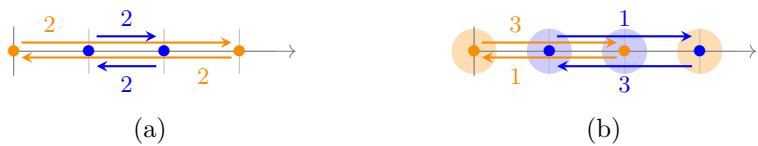

As a second example, consider the dynamically equivalent systems in \Cref{fig:exchange}. Each of the linkage classes of (a), 
    \eq{ 
        V_1 = \{ \cf{0}, \, \cf{3X}  \}  
        \quad \text{and} \quad  
        V_2 = \{ \cf{X}, \, \cf{2X} \},
    }
is split between different linkage classes of (b), with $\cf{0}$ and $\cf{2X}$ belonging to one linkage class, while $\cf{X}$ and $\cf{3X}$ belonging to another. Note that the systems shown in \Cref{fig:exchange} have linkage classes that generate linearly \emph{dependent} stoichiometric subspaces. The two networks also have deficiency one.

The examples in \Cref{fig:contain,fig:exchange} illustrate that, at least in principle, that vertices can be arranged into different linkage classes while maintaining dynamical equivalence.

If the partitioning of vertices by linkage classes are identical between two \WRz\ realizations, treating each linkage class as if it is a mass-action system with one connected component, we can conclude uniqueness from \Cref{cor:l=1}. In the remainder of this section, we prove that the situations in \Cref{fig:contain,fig:exchange} are inconsistent with deficiency zero. More precisely, \Cref{lem:contain} shows that if a linkage class is properly contained in a linkage class of another realization (the situation in \Cref{fig:contain-b}), then the latter linkage class cannot be affinely independent. \Cref{lem:exchange} shows that if a linkage class is split between different linkage classes of another realization (the scenario of \Cref{fig:exchange}), then the stoichiometric subspaces of the latter's linkage classes cannot be linearly independent.

Although the following lemmas repeatedly refer to a reaction network $G$ with one linkage class, we have in mind $G$ as one connected component of a larger reaction network.

\begin{lem}
\label{lem:contain}
    Let $(G,\vv\kk)$ be a mass-action system with one linkage class. Let $(G',\vv\kk')$ be a weakly reversible mass-action system with one linkage class such that $V \subsetneq V'$. If they are dynamically equivalent on $V$, the vertices of $G'$ cannot be affinely independent. 
\end{lem}       
\begin{proof}
Since $G'$ is strongly connected and $V \subsetneq V'$, there exists $\yy_0 \in V$ and $\yy'\in V'\setminus V$ such that $\yy_0 \to \yy'\in E'$. The rate constant for this reaction in $(G', \vv\kk')$ is non-zero. Dynamical equivalence at $\yy_0$ demands that 
    \eq{ 
        \vv 0 &= \sum_{\yy'_i \in V'\setminus V} \! \kk'_i (\yy'_i - \yy_0) 
        + \sum_{\yy_i \in V} (\kk'_i - \kk_i) (\yy_i - \yy_0), 
    }
where $\kk_i$, $\kk'_i \geq 0$ are the appropriate rate constants of $\yy_0 \to \yy_i$ in either $G$ or $G'$. If the vertices of $G'$ are affinely independent, the coefficients in the above equation must be zeroes. In particular, $\kk'_i = 0$ whenever $\yy'_i \not\in V$, contradicting the existence of a reaction in $G'$ from $V$ to $V'\setminus V$.  
\end{proof}

Recall that the stoichiometric subspace $S$ of a reaction network $G$ is the linear span of all reaction vectors in $G$. If $G$ has a single linkage class, then $S = S(V)$, where $S(V)$ is the span of vectors that point between vertices. The following lemma characterizes the stoichiometric subspace of a weakly reversible mass-action system $(G,\vv\kk)$ in terms of the net reaction vectors.

\begin{lem}
\label{lem:netvect} 
    Let $(G,\vv\kk)$ be a mass-action system with one linkage class. For each $\yy_i \in V$, let $\vv w_i$ be the net reaction vector from $\yy_i$, and let $S$ be the stoichiometric subspace. Then we have:
    \begin{enumerate}[label={(\roman*)}]
        \item $\Span \{ \vv w_i\}_{i=1}^m \subseteq S$.  
        \item If $G$ is weakly reversible, then $V$ is the set of \emph{source} vertices, and $\Span\{ \vv w_i\}_{i=1}^m = S$. 
        \item If $G$ is weakly reversible and deficiency zero, then any $m-1$ vectors from the set $\{\vv w_i\}_{i=1}^m$ are linearly independent, i.e., the set is a basis for $S$. 
    \end{enumerate}
\end{lem}
\begin{proof}
By definition, a net reaction vector is  
    \eq{ 
        \vv w_i = \sum_{\yy_j \in V} \kk_{ij} (\yy_j - \yy_i), 
    }
where $\kk_{ij} > 0$ if $\yy_i \to \yy_j \in E$ and $\kk_{ij} = 0$ otherwise. This in turn implies that each $\vv w_i$ is a vector in the stoichiometric subspace $S$. Thus, $\Span \{ \vv w_i\}_{i=1}^m \subseteq S$.
    
Suppose $G$ is weakly reversible, and suppose for a contradiction that the net reaction vectors do not span all of $S$, i.e., $W = \Span \{ \vv w_i\}_{i=1}^m \subsetneq S$. Then there exists a non-zero vector $\vv v \in S$ that is perpendicular to $W$. Since $\vv v \neq \vv 0$, there exists a reaction $\yy_i \to \yy_j \in E$ such that $\vv v \cdot (\yy_j - \yy_i) \neq 0$. In particular, the set $\{ \vv v \cdot \yy_i \}_{i=1}^m$ has at least two different numbers. Let $V_{\max} = \{ \yy_i \in V \st \vv v \cdot \yy_i = \max_j \vv v \cdot \yy_j\}$ be the subset of vertices which maximizes the dot product. Weak reversibility implies that there exists an edge from a vertex in $V_{\max}$ to a  vertex not in it. Without loss of generality, let $\yy_1 \to \yy_2$ be this edge, where $\yy_1 \in V_{\max}$. Note that for all $i=1$, $2,\ldots, m$, we have $\vv v \cdot (\yy_i - \yy_1) \leq 0$, so
    \eq{ 
        \vv v \cdot \vv w_1 
        &= \sum_{\yy_j \in V} \kk_{1j} \vv v \cdot (\yy_j - \yy_1) \leq \kk_{12} \vv v \cdot (\yy_2 - \yy_1) < 0. 
    }
In other words, $\vv v$ is not perpendicular to $\vv w_1$, contradicting our assumption that $\vv v$ is perpendicular to $W$. 

Finally, suppose that $G$ is \WRz. Let $\mm W \in \rr^{n\times m}$ be the matrix with $\vv w_i$ as its $i$th column. By (ii) of the lemma, the range of $\mm W$ is the stoichiometric subspace $S$, which is of dimension $m - 1$ because $\delta = 0 = m - 1 - \dim S$. In other words, the rank of $\mm W$ is $m-1$, thus the dimension of $\ker \mm W$ is 1.

Since $G$ is weakly reversible, the mass-action system $(G,\vv\kk)$ admits a positive steady state $\xx$~\cite{Boros2019}. Rearranging the steady state equation by first summing over the vertices, we obtain the following liner equation involving the net reaction vectors:
    \eq{ 
        \vv 0 
        &= \sum_{\yy_i \in V} \xx^{\yy_i} \!\!\!\sum_{\yy_i \to \yy_j \in E} \kk_{ij} (\yy_j - \yy_i) 
        = \sum_{\yy_i \in V} \xx^{\yy_i} \vv w_i. 
    }
In particular, the vector $\vv \alpha = (\xx^{\yy_1}, \xx^{\yy_2}, \ldots, \xx^{\yy_m})^\top$ has strictly positive coordinates and spans $\ker \mm W$. Therefore, any non-zero vector in $\ker \mm W$ cannot have a zero as one of its components. This implies that any choice of $m-1$ columns of $\mm W$ form a linearly independent set, which spans $\ran \mm W = S$. 
\end{proof}

\begin{lem}
\label{lem:exchange}
    Let $(G, \vv\kk)$ be a $\textit{WR}_\textit{0}$ realization with one linkage class. Let $(G',\vv\kk')$ be a mass-action system with $V \subseteq V'$, and suppose that the vertices in $V$ are split between at least two linkage classes of $G'$.  If these systems are dynamically equivalent on $V$, then the stoichiometric subspaces of the linkage classes of $G'$ cannot be linearly independent.
\end{lem} 
\begin{proof}
Partition $V' = V'_1 \sqcup V'_2 \sqcup \cdots \sqcup V'_\ell$ according to the linkage classes of $G'$. For each $p=1$, $2,\ldots, \ell$, let $V_p = V'_p \cap V$. Without loss of generality (by throwing away linkage classes that do not intersect $V$), we may assume that each $V_p \neq \emptyset$. For each $\yy_i \in V$, let $\vv w_i$ be its net reaction vector, and let $S_p = \Span\{ \vv w_i \st \yy_i \in V_p\}$ be the span of net reaction vectors in $V_p$. By  \hyperref[lem:netvect]{\Cref{lem:netvect}(i)} on the connected component $V'_p$, we know that $S_p \subseteq S(V'_p)$. However, applying \hyperref[lem:netvect]{\Cref{lem:netvect}(iii)} to all of $V$ implies that $S_1$ is not linearly independent of $S_2 + S_3 + \cdots + S_\ell$, i.e., $S(V'_1)$,  $S(V'_2), \ldots, S(V'_\ell)$ are not linearly independent. 
\end{proof}

We now prove uniqueness of \WRz\ realizations for networks with any number of linkage classes. 

\begin{thm}
\label{thm:WR}
    Any weakly reversible, deficiency zero realization is unique. In other words, if  a mass-action system admits several different realizations, then at most one of them can be $\textit{WR}_\textit{0}$.
\end{thm}
\begin{proof}
    Let $(G,\vv\kk)$ be a \WRz\ realization. Suppose that $(G',\vv\kk')$ is a dynamically equivalent \WRz\ realization. As noted in \Cref{prop:V,prop:ell}, the vertices of \WRz\ realizations are determined by the monomials that appear explicitly in the associated system of differential equations, and the number of linkage classes remains constant. Moreover, each linkage class has zero deficiency. Partition the vertices by the linkage classes of $G$ as $V = V_1 \sqcup V_2 \sqcup \cdots \sqcup V_{\ell}$. Similarly partition the vertices by the linkage classes of $G'$, as in $V = V'_1 \sqcup V'_2 \sqcup \cdots \sqcup V'_\ell$. If any of linkage classes share the same set of vertices, i.e., $V_i = V'_j$ for some $i$, $j$, it follows from Corollary \ref{cor:l=1} that this linkage class in $(G,\vv\kk)$ is identical in structure and rate constants to that of $(G', \vv\kk')$. We proceed by induction on the number of differing linkage classes. Without loss of generality, we assume that $\ell \geq 2$ and there is no identical linkage classes between the two realization. 
    
    Recall from \Cref{thm:def0} that any deficiency zero realization has affinely independent linkage classes, and the stoichiometric subspaces of the linkage classes are linearly independent. The linkage class, say $V_1$, of $G$ intersects non-trivially with some linkage class of $G'$, say $V'_j$, in the sense that either $V'_j \setminus V_1 \neq \emptyset$ or $V_1 \setminus V'_j \neq \emptyset$ (or both). It is neither the case that $V_1 \subsetneq V'_j$ nor $V'_j \subsetneq V_1$, or we would contradict the affine independence assumption of the larger subset of vertices by Lemma~\ref{lem:contain}. Hence it must be the case that the intersection is non-trivial, i.e., $V_1 \cap V'_j \neq \emptyset$, and the symmetric difference is also non-trivial, i.e., $V_1 \triangle V'_j \neq \emptyset$. This falls under the setup of \Cref{lem:exchange}; thus $S(V'_j)$ is not linearly independent from $S(V'\setminus V'_j)$. In other words, the stoichiometric subspaces of the linkage classes of $G'$ are not linearly independent and $G'$ has positive deficiency. Therefore, we conclude that $V_1 = V'_j$, a contradiction.
\end{proof}

\section{Discussion}
\label{sec:concl}

In this paper, we proved that any weakly reversible, deficiency zero (\WRz) realization of a mass-action system is unique, as conjectured in \cite{CraciunJohnstonSzederkenyiTonelloTothYu2020}. 
Since deficiency zero weakly reversible networks are {\em minimal} representations of mass-action systems, this provides a possible {\em Occam's razor} approach to network identification and parameter identification, since neither one of these identification problems has a unique solution in general~\cite{CraciunPantea2008}. 

In future work~\cite{paper_3_part_2}, we will use some of the approaches developed here to design an efficient algorithm for the identification of these networks. 
A similar approach may be used for the identification of network representations of lowest deficiency, and allow for wider applicability of classical results for networks with positive deficiency, such as the Deficiency One Theorem~\cite{feinberg2019foundations}.  

\section*{Acknowledgements} 

The authors were supported in part by the National Science Foundation under grant DMS--1816238. G.C. was also partially supported by a Simons Foundation fellowship, and P.Y.Y. was also partially supported by the NSERC.  
\bibliographystyle{siam}
\bibliography{cit}

\begin{bibdiv}
\begin{biblist}

\bib{Anderson2011_GAC}{article}{
      author={Anderson, David~F.},
       title={A proof of the global attractor conjecture in the single linkage
  class case},
        date={2011},
     journal={SIAM Journal on Applied Mathematics},
      volume={71},
      number={4},
       pages={1487\ndash 1508},
}

\bib{Boros2019}{article}{
      author={Boros, Bal\'azs},
       title={Existence of positive steady states for weakly reversible
  mass-action systems},
        date={2019},
     journal={SIAM Journal on Mathematical Analysis},
      volume={51},
      number={1},
       pages={435\ndash 449},
}

\bib{boros2019weakly}{article}{
      author={Boros, Bal\'azs},
      author={Craciun, Gheorghe},
      author={Yu, Polly~Y.},
       title={Weakly reversible mass-action systems with infinitely many
  positive steady states},
        date={2020},
     journal={SIAM Journal on Applied Mathematics},
      volume={80},
      number={4},
       pages={1936\ndash 1946},
}

\bib{BorosHofbauer2019_GAC}{article}{
      author={Boros, Bal\'azs},
      author={Hofbauer, Josef},
       title={Permanence of weakly reversible mass-action systems with a single
  linkage class},
        date={2020},
     journal={SIAM Journal on Applied Dynamical Systems},
      volume={19},
      number={1},
       pages={352\ndash 365},
}

\bib{BrustengaCraciunSorea2020}{article}{
      author={Brustenga~i Moncus\'i, Laura},
      author={Craciun, Gheorghe},
      author={Sorea, Miruna-Stefana},
       title={Disguised toric dynamical systems},
        date={2020},
      eprint={arXiv:2006.01289},
}

\bib{Craciun2019_GAC-inclusion}{article}{
      author={Craciun, Gheorghe},
       title={Polynomial dynamical systems, reaction networks, and toric
  differential inclusions},
        date={2019},
     journal={SIAM Journal on Applied Algebra and Geometry},
      volume={3},
      number={1},
       pages={87\ndash 106},
}

\bib{CraciunDickensteinShiuSturmfels2009}{article}{
      author={Craciun, Gheorghe},
      author={Dickenstein, Alicia},
      author={Shiu, Anne},
      author={Sturmfels, Bernd},
       title={Toric dynamical systems},
        date={2009},
     journal={Journal of Symbolic Computation},
      volume={44},
      number={11},
       pages={1551\ndash 1565},
}

\bib{paper_3_part_2}{article}{
      author={Craciun, Gheorghe},
      author={Jin, Jiaxin},
      author={Yu, Polly~Y.},
       title={Efficient identification of weakly reversible and deficiency zero
  realizations},
     journal={in preparation},
}

\bib{CraciunJinYu2019}{article}{
      author={Craciun, Gheorghe},
      author={Jin, Jiaxin},
      author={Yu, Polly~Y.},
       title={An efficient characterization of complex-balanced,
  detailed-balanced, and weakly reversible systems},
        date={2020},
     journal={SIAM Journal on Applied Mathematics},
      volume={80},
      number={1},
       pages={183\ndash 205},
}

\bib{CraciunJinYu_STN}{article}{
      author={Craciun, Gheorghe},
      author={Jin, Jiaxin},
      author={Yu, Polly~Y.},
       title={Single-target networks},
        date={2020},
     journal={Discrete and Continuous Dynamical Systems-B},
      volume={doi:10.3934/dcdsb.2021065},
      eprint={arXiv:2006.01192},
}

\bib{CraciunJohnstonSzederkenyiTonelloTothYu2020}{article}{
      author={Craciun, Gheorghe},
      author={Johnston, Matthew~D.},
      author={Szederk\'enyi, G\'abor},
      author={Tonello, Elisa},
      author={T\'oth, J\'anos},
      author={Yu, Polly~Y.},
       title={Realizations of kinetic differential equations},
        date={2020},
     journal={Mathematical Biosciences and Engineering},
      volume={17},
      number={1},
       pages={862\ndash 892},
         url={https://www.aimspress.com/fileOther/PDF/MBE/mbe-17-01-046.pdf},
}

\bib{CraciunNazarovPantea2013_GAC}{article}{
      author={Craciun, Gheorghe},
      author={Nazarov, Fedor},
      author={Pantea, Casian},
       title={Persistence and permanence of mass-action and power-law dynamical
  systems},
        date={2013},
     journal={SIAM Journal on Applied Mathematics},
      volume={73},
      number={1},
       pages={305\ndash 329},
}

\bib{CraciunPantea2008}{article}{
      author={Craciun, Gheorghe},
      author={Pantea, Casian},
       title={Identifiability of chemical reaction networks},
        date={2008},
     journal={Journal of Mathematical Chemistry},
      volume={44},
       pages={244\ndash 259},
}

\bib{CraciunSorea2020}{article}{
      author={Craciun, Gheorghe},
      author={Sorea, Miruna-Stefana},
       title={The structure of the moduli spaces of toric dynamical systems},
        date={2020},
      eprint={arXiv:2008.11468},
}

\bib{Csercsik2011ParametricUO}{article}{
      author={Csercsik, D\'avid},
      author={Szederk\'enyi, G\'abor},
      author={Hangos, Katalin~M.},
       title={Parametric uniqueness of deficiency zero reaction networks},
        date={2012},
     journal={Journal of Mathematical Chemistry},
      volume={50},
       pages={1\ndash 8},
}

\bib{feinberg1972complex}{article}{
      author={Feinberg, Martin},
       title={Complex balancing in general kinetic systems},
        date={1972},
     journal={Archive for Rational Mechanics and Analysis},
      volume={49},
      number={3},
       pages={187\ndash 194},
}

\bib{Feinberg1987}{article}{
      author={Feinberg, Martin},
       title={Chemical reaction network structure and the stability of complex
  isothermal reactors--{I}. the deficiency zero and deficiency one theorems},
        date={1987},
     journal={Chemical Engineering Science},
      volume={42},
      number={10},
       pages={2229\ndash 2268},
}

\bib{feinberg2019foundations}{book}{
      author={Feinberg, Martin},
       title={Foundations of chemical reaction network theory},
   publisher={Springer},
        date={2019},
}

\bib{gopalkrishnan2014geometric}{article}{
      author={Gopalkrishnan, Manoj},
      author={Miller, Ezra},
      author={Shiu, Anne},
       title={A geometric approach to the global attractor conjecture},
        date={2014},
     journal={SIAM Journal on Applied Dynamical Systems},
      volume={13},
      number={2},
       pages={758\ndash 797},
}

\bib{horn1972necessary}{article}{
      author={Horn, Fritz},
       title={Necessary and sufficient conditions for complex balancing in
  chemical kinetics},
        date={1972},
     journal={Archive for Rational Mechanics and Analysis},
      volume={49},
      number={3},
       pages={172\ndash 186},
}

\bib{Horn1974_GAC}{article}{
      author={Horn, Fritz},
       title={The dynamics of open reaction systems},
        date={1974},
     journal={SIAM-AMS Proceedings},
      volume={8},
       pages={125\ndash 137},
}

\bib{HornJackson1972}{article}{
      author={Horn, Fritz},
      author={Jackson, Roy},
       title={General mass action kinetics},
        date={1972},
     journal={Archive for Rational Mechanics and Analysis},
      volume={47},
      number={2},
       pages={81\ndash 116},
}

\bib{JohnstonSiegelSzederkenyi2013}{article}{
      author={Johnston, Matthew~D.},
      author={Siegel, David},
      author={Szederk\'enyi, G\'abor},
       title={Computing weakly reversible linearly conjugate chemical reaction
  networks with minimal deficiencys},
        date={2013},
     journal={Mathematical Biosciences},
      volume={241},
      number={1},
       pages={88–\ndash 98},
}

\bib{Malthus1798}{book}{
      author={Malthus, Thomas},
       title={{An Essay on the Principle of Population}},
   publisher={J. Johnson, London},
        date={1798},
}

\bib{Mckeithan1995}{article}{
      author={McKeithan, Timothy~W.},
       title={Kinetic proofreading in t-cell receptor signal transduction},
        date={1995},
     journal={Proceedings of the National Academy of Sciences of the United
  States of America},
      volume={92},
      number={11},
       pages={5042\ndash 5046},
}

\bib{Pantea2012_GAC}{article}{
      author={Pantea, Casian},
       title={On the persistence and global stability of mass-action systems},
        date={2012},
     journal={SIAM Journal on Mathematical Analysis},
      volume={44},
      number={3},
       pages={1636\ndash 1673},
}

\bib{RudanSzederkenyiKatalinPeni2014}{article}{
      author={Rudan, J\'anos},
      author={Szederk\'enyi, G\'abor},
      author={Hangos, Katalin~M.},
      author={P\'eni, Tam\'as},
       title={Polynomial time algorithms to determine weakly reversible
  realizations of chemical reaction networks},
        date={2014},
     journal={Journal of Mathematical Chemistry},
      volume={52},
       pages={1386–\ndash 1404},
}

\bib{Johnston_note}{article}{
      author={Siegel, David},
      author={Johnston, Matthew~D.},
       title={Linearization of complex balanced chemical reaction systems},
        date={2008},
     journal={unpublished note},
        note={Available at
  \url{https://johnstonmd.files.wordpress.com/2015/02/linearization.pdf}},
}

\bib{Sontag2001}{article}{
      author={Sontag, Eduardo~D.},
       title={Structure and stability of certain chemical networks and
  applications to the kinetic proofreading model of t-cell receptor signal
  transduction},
        date={2001},
     journal={IEEE Transactions on Automatic Control},
      volume={46},
      number={7},
       pages={1028\ndash 1047},
}

\bib{Szederkenyi2012_Toolbox}{article}{
      author={Szederk\'enyi, G\'abor},
      author={Banga, Julio~R.},
      author={Alonso, Antonio~A.},
       title={{CRNreals}: a toolbox for distinguishability and identifiability
  analysis of biochemical reaction networks},
        date={2012},
     journal={Bioinformatics},
      volume={28},
      number={11},
       pages={1549\ndash 1550},
}

\bib{Verhulst1838}{article}{
      author={Verhulst, Pierre-Fran{\c c}ois},
       title={Notice sur la loi que la population suit dans son accroissement},
        date={1838},
     journal={Correspondance mathématique et physique},
      volume={10},
       pages={113\ndash 121},
}

\bib{voit2015-150}{article}{
      author={Voit, Eberhard~O},
      author={Martens, Harald~A},
      author={Omholt, Stig~W},
       title={150 years of the mass action law},
        date={2015},
     journal={PLoS Comput Biol},
      volume={11},
      number={1},
       pages={e1004012},
}

\bib{yu2018mathematical}{article}{
      author={Yu, Polly~Y.},
      author={Craciun, Gheorghe},
       title={Mathematical analysis of chemical reaction systems},
        date={2018},
     journal={Israel Journal of Chemistry},
      volume={58},
      number={6-7},
       pages={pp. 733\ndash 741},
}

\end{biblist}
\end{bibdiv}

\end{document}